\documentclass[12pt,reqno]{amsart}
\usepackage[a4paper, left=3cm, right=3cm, bottom=2.5cm]{geometry}
\usepackage{graphicx}
\usepackage{amsfonts}
\usepackage{amsmath}
\usepackage{amssymb}
\usepackage{mathtools}
\usepackage{hyperref}
\usepackage{tikz}
\usetikzlibrary{decorations.pathreplacing}
\usepackage{nicematrix}
\usepackage{graphics}
\usepackage[misc]{ifsym}
\usepackage{bbding}
\usepackage{todonotes}
\usepackage{blkarray}
\usepackage{paralist}
\usepackage[normalem]{ulem}
\usepackage[section]{algorithm}
\usepackage[noend]{algpseudocode}
\usepackage{blindtext}

\newcommand{\zz}{\mathbb{Z}}
\newcommand{\qq}{\mathbb{Q}}

\newcommand{\pp}{\mathbb{P}}
\newcommand{\rr}{\mathbb{R}}
\newcommand{\cc}{\mathbb{C}}

\newcommand{\C}{\mathcal{C}}
\newcommand{\D}{\mathcal{D}}
\renewcommand{\H}{\mathcal{H}}
\newcommand{\M}{\mathcal{M}}
\renewcommand{\P}{\mathcal{P}}
\newcommand{\Q}{\mathcal{Q}}
\newcommand{\T}{\mathcal{T}}
\newcommand{\V}{\mathcal{V}}

\newcommand{\im}{\mathrm{Im}}

\newcommand{\hh}{\mathfrak{H}}
\newcommand{\xx}{\mathbf{x}}

\newcommand{\E}{\mathcal{E}}

\renewcommand{\L}{\mathcal{L}}

\newcommand{\N}{\mathcal{N}}
\newcommand{\R}{\mathcal{R}}

\newcommand{\indep}{\perp\!\!\!\perp}
\newcommand{\ket}[1]{|#1\rangle}
\newcommand{\bra}[1]{\langle#1|}
\newcommand{\Id}{\mathrm{Id}}
\newcommand{\Sym}{\mathbb{S}}
\newcommand{\qcmi}[3]{I(#1 : #2 \,|\, #3)}
\newcommand{\dd}{\mathfrak{D}}

\DeclareMathOperator{\proj}{Proj}
\DeclareMathOperator{\tr}{Tr}
\DeclareMathOperator{\diag}{diag}
\DeclareMathOperator{\supp}{Supp}
\DeclareMathOperator{\gm}{GM}
\DeclareMathOperator{\gv}{GV}
\DeclareMathOperator{\argmax}{argmax}
\DeclareMathOperator{\argmin}{argmin}

\newcommand{\minus}{\scalebox{0.75}[1.0]{$-$}}

\newtheorem{theorem}{Theorem}[section]
\newtheorem{lemma}[theorem]{Lemma}
\newtheorem{proposition}[theorem]{Proposition}

\theoremstyle{definition}
\newtheorem{definition}[theorem]{Definition}
\newtheorem{example}[theorem]{Example}

\newtheorem{question}[theorem]{Question}

\newtheorem{problem}[theorem]{Problem}
\newtheorem{construction}[theorem]{Construction}

\theoremstyle{remark}
\newtheorem{remark}[theorem]{Remark}
\newtheorem*{claim}{Claim}

\numberwithin{equation}{section}
\title{Algebraic Geometry of Quantum Graphical Models}

\author{Eliana Duarte}
\address{Universidade do Porto, Rua do Campo Alegre 687, 4169-007 Porto, Portugal}
\email{eliana.gelvez@fc.up.pt}

\author{Dmitrii Pavlov}
\address{Max Planck Institute for Mathematics in the Sciences, Inselstrasse 22, 04103 Leipzig, Germany}
\email{dmitrii.pavlov@mis.mpg.de}

\author{Maximilian Wiesmann}
\address{Max Planck Institute for Mathematics in the Sciences, Inselstrasse 22, 04103 Leipzig, Germany}
\email{wiesmann@mis.mpg.de}

\begin{document}

\maketitle

\begin{abstract}
    Algebro-geometric methods have proven to be very successful in the study of graphical models in statistics. In this paper we introduce the foundations to carry out a similar study of their quantum counterparts. These quantum graphical models are families of quantum states satisfying certain locality or correlation conditions encoded by a graph. We lay out several ways to associate an algebraic variety to a quantum graphical model. The classical graphical models can be recovered from most of these varieties by restricting to quantum states represented by diagonal matrices. We study fundamental properties of these varieties and provide algorithms to compute their defining equations. Moreover, we study quantum information projections to quantum exponential families defined by graphs and prove a quantum analogue of Birch's Theorem.
\end{abstract}



\section{Introduction}

The goal of this paper is to consider \emph{quantum graphical models} \cite{leifer2008quantum} from the point of view of algebraic geometry with the aim of offering a new perspective on open problems in quantum 
information theory. Roughly speaking, when passing from the classical to the quantum setting, we replace probability distributions
with \emph{density matrices}, with the classical case being recovered when these matrices are diagonal. The graph describes a physical quantum system with nodes representing 
subsystems. Such models have also been coined \emph{quantum Markov networks} in the quantum 
information theory literature \cite{brownQMNandCommHams, di2020recoverability, 
poulin2011markov} and have applications to quantum many-body systems, quantum error correction and 
the study of entanglement. 
This article describes different approaches to obtain an algebraic variety associated to a quantum graphical model.\par 

In algebraic statistics, \emph{graphical models} \cite{lauritzenGraphical} play a prominent role, with applications to, among others, phylogenetics, causal inference and medical diagnosis \cite{koller2009,maathuis2018}. Such a model arises from a graph imposing certain \emph{conditional independence} statements on random variables represented by nodes in the graph. As an example \cite[Ex.\ 1.29]{pachter2005algebraic}, consider the chain graph $G$ on three vertices

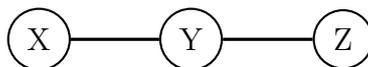
\begin{center}
    \begin{tikzpicture}
        \begin{scope}[every node/.style={circle,thick,draw}]
            \node (X) at (0,0) {X};
            \node (Y) at (2,0) {Y};
            \node (Z) at (4,0) {Z};
        \end{scope}

        \begin{scope}[
              every edge/.style={draw=black,very thick}]
            \path [-] (X) edge (Y);
            \path [-] (Y) edge (Z);
        \end{scope}
    \end{tikzpicture}
\end{center}
with binary random variables $X,Y$ and $Z$. Then $G$ encodes the conditional independence statement $X \indep Z \mid Y$ giving rise to a statistical model described by the algebraic variety
\[
    \M_G = \V(p_{001}p_{100} - p_{000}p_{101}, p_{011}p_{110} - p_{010}p_{111}) \subseteq \pp^7 = \proj(\cc[p_{000},\dots,p_{111}]).
\]
More generally, graphical models for discrete and Gaussian random variables are algebraic varieties. This algebraic perspective advances both the theoretical foundations for these statistical models and the development of new computational methods for the practical use. At the core of these 
advances lies the understanding of the implicit and parametric 
descriptions of the model and its likelihood geometry. This proved to be useful in
model selection, causal discovery, and  maximum likelihood estimation \cite{evans2020,gms2006,shaowei2014,uhler2013}.\par

Motivated by this, we find ways to associate algebraic varieties to quantum graphical models and make progress towards understanding their implicit and parametric descriptions. This leads to a number of interesting varieties and new computational challenges. The role of the maximum likelihood estimator is taken by the \emph{quantum information projection}. We study its geometry for \emph{quantum exponential families} of commuting Hamiltonians, e.g.\ Hamiltonians arising in the context of \emph{graph states}. \par 

The paper is organised as follows. In Section \ref{sec:basic}, we introduce some basic notions of 
quantum information theory. In Section~\ref{sec:qmt}, we introduce the \emph{quantum conditional mutual
information (QCMI) variety} and the \emph{Petz variety}. The former is obtained from studying the structure of quantum states satisfying \emph{strong subadditivity} with equality \cite{hayden2004structure}. The latter is related to the Petz recovery map \cite{petz1986sufficient} and the solution of the Quantum Marginal Problem for the $3$-chain \cite{tyc2015quantum}. In both instances, the graph imposes quantum conditional independence statements, in direct analogy to the classical case. In Section \ref{sec:QGMandGibbs}, we suggest a notion of a
quantum graphical model as the \emph{Gibbs manifold} \cite{pavlov2022gibbs} of a certain family of Hamiltonians and consider the smallest variety that contains it (known as the \emph{Gibbs variety}). Here, the graph encodes a locality structure imposed on the Hamiltonians \cite{brownQMNandCommHams}. This section also includes new results on Gibbs varieties. Finally, we present results on quantum exponential families coming from \emph{stabiliser codes} \cite[\S 10]{nielsen2002quantum}; one particular example are families of Hamiltonians associated to graph states \cite{graphStates}. We study the quantum information projection \cite{niekamp2013computing} and relate it to the classical theory of maximum likelihood estimation, proving a generalisation of Birch's Theorem.\par 
Throughout the sections we provide algorithms to compute the varieties appearing in our study and present computational examples. Those are implemented in \texttt{Julia} making use of the computer algebra package \texttt{Oscar.jl} \cite{OSCAR} and the numerical algebraic geometry tool \texttt{HomotopyContinuation.jl} \cite{htpy2018}; it is available at \cite{mathrepo}.

\section{Basic Notions of Quantum Information Theory} \label{sec:basic}
In this section we collect some fundamental notions of quantum information theory. The reader is referred to \cite{landsberg2019},\cite[Part III]{nielsen2002quantum} for a much more detailed introduction to the subject.\par 
A \emph{quantum state} on $N$ \emph{qudits} is represented by a vector $\ket{\psi}\in \H = \H_1\otimes\dots\otimes\H_N$ of unit length, where $\H_i$ is the Hilbert space $\H_i\cong \cc^d,~ i=1,\dots,N$. Here, we make use of the Dirac notation, i.e.\ $\ket{\psi}$ denotes a column vector and $\bra{\psi}$ its complex conjugate transpose. In the case $N=1$ and $d=2$, $\ket{\psi}$ is called a \emph{qubit} and this will be our primary focus. \par 
In this paper we often consider a simple, undirected graph $G=(V,E)$ on $N$ (ordered) vertices. The graph $G$ represents a quantum system with Hilbert space $\H$ and each vertex $v_i\in V$ corresponds to a physical subsystem in $\H_i\cong \cc^{d_i}$. \par 
An ensemble of quantum states is a collection $\{p_i, \ket{\psi_i}\}_i$ where $\{p_i\}_i$ is a discrete probability distribution. Such an ensemble is described by its \emph{density matrix}
\[
    \rho = \sum_i p_i \ket{\psi_i}\bra{\psi_i}.
\]
Equivalently, we can characterise density matrices as positive semidefinite endomorphisms on $\H$ with unit trace. From now on, we will use the terms ``quantum state'' and ``density matrix'' interchangeably. The set of all density matrices on $\H$ is denoted by $\D(\H)$.\par 

Let $\rho_{AB}$ be a bipartite quantum state, i.e.\ a density operator on $\H_A\otimes \H_B$. We define the \emph{partial trace} over the $B$-system on elementary tensors via
\[
    \tr_B(\ket{a_i}\bra{a_j}\otimes \ket{b_k}\bra{b_l}) \coloneqq \ket{a_i}\bra{a_j}\cdot \tr(\ket{b_k}\bra{b_l}) = \ket{a_i}\bra{a_j}\cdot \langle b_l| b_k\rangle,
\]
where $\ket{a_i},\ket{a_j}\in\H_A$ and $\ket{b_k},\ket{b_l}\in \H_B$ and extend this operation linearly to $\rho_{AB}$. Note that $\tr_B\rho_{AB}$ is a density operator on $\H_A$ and therefore we use the notation $\rho_A\coloneqq \tr_B\rho_{AB}$. One can think of the partial trace operation as the quantum analogue of marginalisation in statistics. Physically speaking, $\rho_A$ describes the state of the subsystem $A$ of the composite system $AB$.

\begin{example}
    Set $\ket{0} \coloneqq (1,0)^T, \ket{1} \coloneqq (0,1)^T$ as a basis of $\cc^2$ (this basis is called \emph{computational basis} in the context of quantum information theory); we also adapt the notation to write $\ket{ij}$ for $\ket{i}\otimes\ket{j},~ i,j\in\{0,1\}$. Consider the \emph{Bell state}
    \[
        \scalebox{0.94}{$\rho_{AB} = \left(\frac{1}{\sqrt{2}}(\ket{00} + \ket{11})\right)\left(\frac{1}{\sqrt{2}}(\bra{00} + \bra{11})\right) = \frac{1}{2}(\ket{00}\bra{00} + \ket{00}\bra{11} + \ket{11}\bra{00} + \ket{11}\bra{11})$}
    \]
    on $\cc^2\otimes \cc^2$. Then the partial trace over $B$ is computed as
    \[
        \scalebox{0.93}{$\tr_B\rho_{AB} = \frac{1}{2}(\ket{0}\bra{0} \langle 0|0\rangle + \ket{0}\bra{1} \langle 1|0\rangle + \ket{1}\bra{0} \langle 0|1\rangle + \ket{1}\bra{1} \langle 1|1\rangle) = \frac{1}{2}(\ket{0}\bra{0} + \ket{1}\bra{1}) = \frac{\Id_2}{2},$}
    \]
    called the \emph{maximally mixed state}, meaning $\rho_{AB}$ is \emph{maximally entangled} \cite[\S 2]{nielsen2002quantum}.
\end{example}

Finally, when we speak of a \emph{Hamiltonian} $H$, we simply mean a real symmetric matrix. The reason we restrict to this class (and do not consider Hermitian matrices as would be natural in many contexts in quantum physics) is that symmetric matrices form a linear space, while the set of Hermitian matrices is not an algebraic variety. Likewise, we consider density matrices to be real positive semidefinite (PSD). Note that $\exp(H)$ is positive definite and thus can (up to normalisation) be regarded as a quantum state.\par 
The ambient space of the algebraic varieties we consider is $\Sym^n\cong \cc^{n(n+1)/2}$, each point being a complex symmetric matrix. To recover a specific quantum model, we intersect the variety with the PSD cone and the hyperplane of trace one matrices.

\section{Quantum Graphical Models on Trees} \label{sec:qmt}

It is an important and mostly open problem in quantum information theory to describe the set of compatible density matrices on subsystems of a composite system. This is known as the Quantum Marginal Problem. 
\begin{problem}[Quantum Marginal Problem]
    Let $S=\{1,\dots,N\}$ be a composite system on $N$ qudits and suppose we are given density matrices $\rho_{S_1},\dots,\rho_{S_n}$ of $n$ subsystems $S_1,\dots,S_n\subseteq S$. What conditions do $\rho_{S_1},\dots,\rho_{S_n}$ have to satisfy to arise from $\rho_S$ as $\rho_{S_i} = \tr_{S\setminus S_i}\rho_S$?
\end{problem}

So far, for general graphs, this problem has only been solved in the case of disjoint subsystems $S_i$. See \cite{tyc2015quantum} for a survey. However, for trees it is possible to reconstruct a quantum state from its two-body marginals \cite{di2021complexity,di2020recoverability}. This can be done using algebraic methods and motivates the algebro-geometric notions of quantum graphical models we introduce in this section. Associating algebraic varieties to quantum graphical models and studying their defining equations might open a new way of attacking this problem.\par

\subsection{Quantum conditional mutual information varieties}
In this section we give the definition of quantum conditional mutual information (QCMI) and collect some of its properties. The vanishing of QCMI should be thought of as a quantum analogue to conditional independence and gives rise to an algebraic variety that we call the \emph{QCMI variety}. 


The \emph{von Neumann entropy} $S(\rho)$ of a quantum state $\rho$ is $S(\rho)\coloneqq -\tr(\rho\log\rho)$ and is a straightforward generalisation of the classical Shannon entropy; here, the logarithm has base two. Let $\rho_{ABC}$ be a tripartite state; then the \emph{quantum conditional mutual information} between $A$ and $C$ given $B$ is
\[
    I(A:C\,|\,B)\coloneqq S(AB) + S(BC) - S(ABC) -S(B),
\]
where $S(ABC)=S(\rho_{ABC}),S(AB)=S(\tr_C\rho_{ABC})$ etc. Note that if one replaces the von Neumann entropy with the Shannon entropy in the above definition, one obtains the classical conditional mutual information $I_{\text{cl}}(A:C\,|\,B)$ between random variables $A$ and $C$ given $B$. Its vanishing $I_{\text{cl}}(A:C\,|\,B)=0$ is well-known to be equivalent to the conditional independence $A\indep C\mid B$ and leads to two possible
different factorisations of the joint probability distribution $p(a,b,c)$ \cite{hayden2004structure}. The vanishing of QCMI of a tripartite system behaves similarly,
implying  a more involved factorisation of the density matrix of the tripartite system (see Construction \ref{rem:qfactorisation}).\par 
The following constitutes a quantum analogue to the conditional independence axioms for probability distributions, see \cite[Thm.\ 4.5]{leifer2008quantum}.

\begin{proposition}
    \label{prop:qcmiAxioms}
    Let $S$ be a composite quantum system with disjoint subsystems $A, B, C, D \subseteq S$. Then the following implications hold:
    \begin{enumerate}
        \item $\qcmi{A}{C}{B} = 0 \Rightarrow \qcmi{C}{A}{B} = 0$ (Symmetry),
        \item $\qcmi{A}{CD}{B} = 0 \Rightarrow \qcmi{A}{C}{B} = 0$ (Decomposition),
        \item $\qcmi{A}{CD}{B} = 0 \Rightarrow \qcmi{A}{C}{BD} = 0$ (Weak Union),
        \item $\qcmi{A}{B}{CD} = 0 \text{ and } \qcmi{A}{D}{C} = 0 \Rightarrow \qcmi{A}{BD}{C} = 0$ (Contraction).
    \end{enumerate}
\end{proposition}

The QCMI is closely related to the celebrated \emph{strong subadditivity} (SSA) inequality \cite{lieb1973proof}
\[
    S(ABC) + S(B) \leq S(AB) + S(BC).
\]
The case of equality in SSA, i.e.\ $I(A:C\,|\,B)=0$, has been intensively studied; the main result is the following Theorem from \cite{hayden2004structure}.

\begin{theorem}\label{thm:central}
    A quantum state $\rho_{ABC}$ on $\H_A\otimes \H_B\otimes\H_C$ satisfies SSA with equality if and only if there exists a decomposition of $\H_B$ as
    \[
        \H_B = \bigoplus_j \H_{b_j^L}\otimes \H_{b_j^R}
    \]
    such that $\rho_{ABC}$ decomposes as 
    \[
        \rho_{ABC} = \bigoplus_j q_j \rho_{Ab_j^L}\otimes \rho_{b_j^RC},
    \]
    where $\{q_j\}_j$ is a probability distribution and $\rho_{Ab_j^L},\rho_{b_j^RC}$ are states on $\H_A\otimes \H_{b^L_j}$ and $\H_{b^R_j}\otimes \H_C$, respectively.
\end{theorem}

\begin{construction}[QCMI variety of the 3-chain graph] \label{rem:qfactorisation}
   The following reformulation of Theorem~\ref{thm:central} plays a central role in the construction of the QCMI variety.
    Setting $\Lambda_{AB}\coloneqq \bigoplus_j q_j\rho_{Ab_j^L}\otimes \Id_{b_j^RC}$ and $\Lambda_{BC}\coloneqq \bigoplus_j \Id_{Ab_j^L}\otimes \rho_{b_j^RC}$, we arrive at
\begin{equation}
    \label{equ:param3path}
    I(A:C\,|\,B)=0 \text{ if and only if } \rho_{ABC} = \Lambda_{AB}\Lambda_{BC} \text{ with } \left[\Lambda_{AB},\Lambda_{BC}\right] = 0,
\end{equation}
where $\Lambda_{AB},\Lambda_{BC}$ are symmetric matrices acting on $\H_A\otimes\H_B\otimes\H_C$ and $\Lambda_{AB}$ acts as identity on $\H_C$, and, likewise, $\Lambda_{BC}$ acts as identity on $\H_A$ \cite{brownQMNandCommHams}. Then the right hand side of (\ref{equ:param3path}) gives rise to a parametrisation of a variety we denote by $\Q_{\qcmi{A}{C}{B}}$.
\end{construction}

\begin{example}[$\Q_{\qcmi{A}{C}{B}}$ in the qubit case]
    \label{eg:qcmiVar3path}
    Let $\H_A\cong\H_B\cong\H_C\cong\cc^2$ and write $\Lambda_{AB} = M\otimes\Id_2,~\Lambda_{BC} = \Id_2\otimes N$ for $M,N\in\Sym^4$. In this case, the parametrisation of $\Q_{\qcmi{A}{C}{B}}$ induced by the right hand side of \eqref{equ:param3path} sends 
    \[
        M=\begin{pmatrix}
            x_1 & x_2 & x_3 & x_4 \\
            x_2 & x_5 & x_6 & x_7 \\
            x_3 & x_6 & x_8 & x_9 \\
            x_4 & x_7 & x_9 & x_{10} \\
        \end{pmatrix} \text{ and  }
        N = \begin{pmatrix}
            y_1 & y_2 & y_3 & y_4 \\
            y_2 & y_5 & y_6 & y_7 \\
            y_3 & y_6 & y_8 & y_9 \\
            y_4 & y_7 & y_9 & y_{10} \\
        \end{pmatrix}
    \]
    to the matrix
    \[
    \scalebox{0.728}{$\begin{pmatrix}
        x_{1} y_{1} + x_{2} y_{4} & x_{1} y_{2} + x_{2} y_{5} & x_{1} y_{4} + x_{2} y_{6} & x_{1} y_{7} + x_{2} y_{9} & x_{4} y_{1} + x_{7} y_{4} & x_{4} y_{2} + x_{7} y_{5} & x_{4} y_{4} + x_{7} y_{6} & x_{4} y_{7} + x_{7} y_{9} \\
        x_{1} y_{2} + x_{2} y_{7} & x_{1} y_{3} + x_{2} y_{8} & x_{1} y_{5} + x_{2} y_{9} & x_{1} y_{8} + x_{2} y_{10} & x_{4} y_{2} + x_{7} y_{7} & x_{4} y_{3} + x_{7} y_{8} & x_{4} y_{5} + x_{7} y_{9} & x_{4} y_{8} + x_{7} y_{10} \\
        x_{2} y_{1} + x_{3} y_{4} & x_{2} y_{2} + x_{3} y_{5} & x_{2} y_{4} + x_{3} y_{6} & x_{2} y_{7} + x_{3} y_{9} & x_{5} y_{1} + x_{8} y_{4} & x_{5} y_{2} + x_{8} y_{5} & x_{5} y_{4} + x_{8} y_{6} & x_{5} y_{7} + x_{8} y_{9} \\
        x_{2} y_{2} + x_{3} y_{7} & x_{2} y_{3} + x_{3} y_{8} & x_{2} y_{5} + x_{3} y_{9} & x_{2} y_{8} + x_{3} y_{10} & x_{5} y_{2} + x_{8} y_{7} & x_{5} y_{3} + x_{8} y_{8} & x_{5} y_{5} + x_{8} y_{9} & x_{5} y_{8} + x_{8} y_{10} \\
        x_{4} y_{1} + x_{5} y_{4} & x_{4} y_{2} + x_{5} y_{5} & x_{4} y_{4} + x_{5} y_{6} & x_{4} y_{7} + x_{5} y_{9} & x_{6} y_{1} + x_{9} y_{4} & x_{6} y_{2} + x_{9} y_{5} & x_{6} y_{4} + x_{9} y_{6} & x_{6} y_{7} + x_{9} y_{9} \\
        x_{4} y_{2} + x_{5} y_{7} & x_{4} y_{3} + x_{5} y_{8} & x_{4} y_{5} + x_{5} y_{9} & x_{4} y_{8} + x_{5} y_{10} & x_{6} y_{2} + x_{9} y_{7} & x_{6} y_{3} + x_{9} y_{8} & x_{6} y_{5} + x_{9} y_{9} & x_{6} y_{8} + x_{9} y_{10} \\
        x_{7} y_{1} + x_{8} y_{4} & x_{7} y_{2} + x_{8} y_{5} & x_{7} y_{4} + x_{8} y_{6} & x_{7} y_{7} + x_{8} y_{9} & x_{9} y_{1} + x_{10} y_{4} & x_{9} y_{2} + x_{10} y_{5} & x_{9} y_{4} + x_{10} y_{6} & x_{9} y_{7} + x_{10} y_{9} \\
        x_{7} y_{2} + x_{8} y_{7} & x_{7} y_{3} + x_{8} y_{8} & x_{7} y_{5} + x_{8} y_{9} & x_{7} y_{8} + x_{8} y_{10} & x_{9} y_{2} + x_{10} y_{7} & x_{9} y_{3} + x_{10} y_{8} & x_{9} y_{5} + x_{10} y_{9} & x_{9} y_{8} + x_{10} y_{10} \\
    \end{pmatrix}$}.
    \]
    This results in a twelve-dimensional variety inside $\Sym^8$ cut out by 735 equations in degrees one to five; the degree of $\Q_{\qcmi{A}{C}{B}}$ is 110. As all of these equations are homogeneous, $\Q_{\qcmi{A}{C}{B}}$ can be considered as a subvariety of $\pp^{35} = \proj(\cc[z_1,\dots,z_{36}])$. Among these equations only two are linear:
    \[
        z_{14} - z_{18} + z_{23} - z_{29}=0,~
z_{12} - z_{16} - z_{25} + z_{31}=0,
    \]
    and just one has degree five:
    \begin{multline*}
        -z_{13}z_{22}z_{29}z_{31}z_{33} - z_{13}z_{22}z_{31}^2z_{35} + z_{13}z_{24}z_{29}^2z_{33} + z_{13}z_{24}z_{29}z_{31}z_{35} + z_{22}^2 z_{29}z_{31}z_{33} \\+ z_{22}^2z_{31}^2z_{35} + z_{22}z_{24}z_{25}z_{29}z_{35} - z_{22}z_{24}z_{25}z_{31}z_{33} -  z_{22}z_{24}z_{29}^2z_{33} - 2z_{22}z_{24}z_{29}z_{31}z_{35} \\+ z_{22}z_{24}z_{31}^2z_{33} - z_{23}z_{24}^2z_{29}z_{35} + z_{23}z_{24}^2z_{31}z_{33} + z_{24}^2z_{29}^2z_{35} - z_{24}^2z_{29}z_{31}z_{33}=0.
    \end{multline*}
    Here the variables $z_1,\ldots,z_{36}$ denote the entries of
a symmetric ($8\times 8$)-matrix written in order starting from left to right and continuing from top to bottom. Note that if you set all non-diagonal entries in $M$ and $N$ to zero, this results in a monomial parametrisation of the classical graphical model of the 3-chain as presented in the introduction.
\end{example}

\begin{proposition}
    The variety $\Q_{\qcmi{A}{C}{B}}$ is irreducible.
\end{proposition}

\begin{proof}
    Under the composition of morphisms
    \begin{align*}
        \mathrm{U}(8)\times \rr^4\times \rr^4 &\rightarrow  \Sym^8\times\Sym^8  \xrightarrow{\text{mult.}} \Sym^8\\
        (U,\lambda,\mu) & \mapsto \!\!\left(U\diag(\lambda_1,\lambda_1,\ldots,\lambda_4,\lambda_4)U^{-1}, U\diag(\lambda_1,\dots,\lambda_4,\lambda_1,\ldots,\lambda_4)U^{-1}\right) \\
       & \quad \quad (M,N) \mapsto M \cdot N,
    \end{align*}
    where ${\lambda= (\lambda_1,\ldots,\lambda_4)}$ and ${\mu= (\mu_1,\ldots,\mu_4)}$,
    $\Q_{\qcmi{A}{C}{B}}$ is the image of an irreducible variety. \end{proof}

In analogy to the classical theory of graphical models, we associate QCMI statements to separations in an undirected graph.
A \emph{separator} between two sets of nodes $A$ and $C$ in a graph $G$ is a set $B$ of nodes such that every path from a node in $A$ to a node in $C$ contains a node in $B$. For classical graphical models on undirected graphs, the Hammersley--Clifford Theorem \cite[Thm.\ 3.9]{lauritzenGraphical} (see also \cite{hc1971}) states that  a positive probability distribution satisfies the conditional independence statements associated to separations in a graph if and only if it factorises according to the graph. 
One might attempt to achieve a similar factorisation theorem for quantum graphical models, however, such description is not available for arbitrary graphs. Nevertheless, there is a ``quantum Hammersley--Clifford Theorem'' for trees.

\begin{theorem}[{\cite[Thm.\ 1]{poulin2011markov}}]
\label{thm:quHC}
    Let $G = (V,E)$ with $V=\{v_1,\dots,v_N\}$ be a tree and let $\rho$ be a positive definite quantum state on a Hilbert space $\H=\H_1\otimes\dots\otimes\H_N$ satisfying all QCMI statements imposed by $G$. Then $\rho$ can be written as the exponential of a sum of local commuting Hamiltonians, i.e.\ $\rho = \exp(H)$ with
    \[
        H=\sum_{C\in\C(G)}h_{C}, \quad [h_{C},h_{C'}]=0 \text{ for all } C,C'\in\C(G),
    \]
where $\C(G)$ is the set of cliques of $G$ and $h_C$ is only nontrivial on the clique $C$, i.e.\ $h_C$ is an endomorphism on $\H$ acting as identity on each $\H_i$ where $v_i\notin C$.
\end{theorem}

This quantum Hammersley--Clifford Theorem is a generalisation of equation \eqref{equ:param3path} to trees. Along with Example \ref{eg:qcmiVar3path}, this suggests the following construction of the QCMI variety of a tree and an associated quantum graphical model.

\begin{construction}[QCMI variety of a tree]
    \label{constr:QCMIVar}
    Let $G=(V,E)$ be an undirected tree with vertices labelled $S_1,\dots,S_N$. Let $\rho_V=\rho_{S_1\dots S_N}$ be a quantum state on $\H_{1}\otimes\dots\otimes\H_{N}$. For each triple of vertices $S_i,S_j,S_k$ such that $S_j$ separates $S_i$ from $S_k$ in $G$, we impose the \emph{QCMI statement} $\qcmi{S_i}{S_k}{S_j}=0$, i.e.\ we require
    \begin{equation}
        \label{equ:qcmiParam}
        \tr_{V\setminus \{S_i,S_j,S_k\}}\rho_V = \Lambda_{S_iS_j}\Lambda_{S_jS_k} \text{ with } \left[\Lambda_{S_iS_j},\Lambda_{S_jS_k}\right] = 0
    \end{equation}
     as in (\ref{equ:param3path}). Moreover, for any two QCMI statements $\qcmi{S_i}{S_k}{S_j}=0=\qcmi{S_{i^{\prime}}}{S_{k^{\prime}}}{S_{j^{\prime}}}$ we impose the \emph{compatibility constraints}
     \begin{equation}
         \label{equ:compConstraints}
         \tr_{\T\setminus (\T\cap \T^{\prime})}\rho_\T = \tr_{\T^{\prime}\setminus (\T\cap \T^{\prime})}\rho_{\T^{\prime}} \text{ where } \T=(S_i,S_j,S_k),~\T^{\prime} = (S_{i^{\prime}},S_{j^{\prime}},S_{k^{\prime}}).
     \end{equation}
     In the qubit case, this construction gives rise to Algorithm \ref{alg:QCMI} whose output is a variety inside $\Sym^{2^N}$. We call this variety the \emph{QCMI variety} associated to $G$ and denote it by $\Q_G$. Algorithm~\ref{alg:QCMI} constructs the  QCMI variety
     by considering the conditions \eqref{equ:qcmiParam}, \eqref{equ:compConstraints}
     as polynomial constraints in the entries of an arbitrary density matrix $\rho$ and of the matrices $\Lambda_{S_iS_j},\Lambda_{S_jS_k}$, then it eliminates the $\Lambda$ parameters. The QCMI variety $\mathcal{Q}_G$ defines a \emph{quantum graphical model} $\mathcal{M}_G$ by restricting to positive semidefinite matrices with trace one inside $\mathcal{Q}_G$. Note that Algorithm \ref{alg:QCMI} and the notion of the QCMI variety generalise straightforwardly to arbitrary qudit systems.
\end{construction}

\begin{algorithm}[h!]
    \caption{Computing the QCMI variety $\Q_G$}
    \label{alg:QCMI}
    \hspace*{-21.5em} \textbf{Input:} A graph $G = (V,E)$\\
    \hspace*{-16em} \textbf{Output:} Polynomials defining $\Q_G\subseteq \Sym^{2^N}$\\
    \begin{algorithmic}[1]
    \State $N\gets \#V$
    \State $\rho_V \gets$ symmetric $(2^N\times 2^N)$-matrix consisting of variables $\rho_{11},\rho_{12},\dots,\rho_{2^N 2^N}$
    \State $\E \gets \emptyset$ initialise list of equations
    \For{every triple of vertices $\T = (S_i,S_j,S_k)$ such that $S_j$ separates $S_i$ from $S_k$ in $G$}
        \State $\Lambda_{S_iS_j} \gets (\lambda_{lm}^{\T}) \otimes \Id_2$ where $(\lambda_{lm}^{\T})$ is a symmetric $(4\times 4)$-matrix of variables
        \State $\Lambda_{S_jS_k} \gets \Id_2\otimes (\mu_{lm}^{\T})$ where $(\mu_{lm}^{\T})$ is a symmetric $(4\times 4)$-matrix of variables
        \State $\E^{\prime} \gets$ entries of $\tr_{V\setminus \T}\rho_V - \Lambda_{S_iS_j}\Lambda_{S_jS_k}$
        \State $\E^{\prime\prime} \gets$ entries of $\left[\Lambda_{S_iS_j},\Lambda_{S_jS_k}\right]$
        \State $\E\gets \E\cup\E^{\prime}\cup\E^{\prime\prime}$
    \EndFor
    \For{every pair of triples of vertices $\T = (S_i,S_j,S_k)$ and $\T^{\prime} = (S_{i^{\prime}}, S_{j^{\prime}}, S_{k^{\prime}})$}
        \State $\E^{\prime\prime\prime}\gets$ entries of $\tr_{\T\setminus (\T\cap \T^{\prime})}\rho_\T - \tr_{\T^{\prime}\setminus (\T\cap \T^{\prime})}\rho_{\T^{\prime}}$
        \State $\E\gets \E\cup \E^{\prime\prime\prime}$
    \EndFor
    \State $I\gets$ ideal generated by $\E$ in $\cc[\rho, \lambda, \mu]$
    \State $J\gets $ elimination ideal $I\cap\cc[\rho]$
    \State \Return a set of generators of $J$
    \end{algorithmic}
\end{algorithm}

\begin{remark}
    Let $G$ be the path graph on three vertices with ordered vertex labels $A,B$ and $C$ and consider the qubit case $\H_A\cong\H_B\cong\H_C\cong\cc^2$. Then $\Q_G$ is the variety $\Q_{\qcmi{A}{C}{B}}$ from Example \ref{eg:qcmiVar3path}. The computations in this example were carried out using Algorithm \ref{alg:QCMI}.
\end{remark}

\begin{remark}
    Note that as we consider trees, it is equivalent to impose QCMI statements on triples of vertices as in Construction \ref{constr:QCMIVar} or to impose a \emph{global quantum Markov property}, in the sense that one imposes the QCMI statement $\qcmi{A}{C}{B}$ for any triple of sets of vertices $A,B,C\subseteq V$ such that $B$ separates $A$ from $C$, as one can derive the latter from the former by using the Weak Union and Contraction axioms from Proposition \ref{prop:qcmiAxioms}.
\end{remark}

\begin{example}
    Consider the claw tree $G$ on four vertices with
    labels $A,B,C,D$, the set of edges $\{\{A,D\},\{B,D\},\{C,D\}\}$, and the corresponding Hilbert spaces    $\H_A\cong\H_B\cong\H_C\cong\H_D\cong\cc^2$. The Hilbert space of the full system is
    $\H=\H_A\otimes\H_B\otimes\H_C\otimes\H_D\cong \cc^{16}$. 
    Every path in $G$ with three vertices imposes a QCMI statement and any such path
    contains the node $D$. The model $\M_G$ consists of all density matrices $\rho$ that satisfy the three QCMI statements
    \[ \qcmi{A}{B}{D}=\qcmi{A}{C}{D}=\qcmi{B}{C}{D}=0.\]
    These lead to  factorisations of the marginal subsystems as
    \begin{align*}
        \tr_{A}\rho_{ABCD} &= \rho_{BCD}= \Lambda_{BD}\Lambda_{CD} \text{ with }[\Lambda_{BD},\Lambda_{CD}]=0,\\
        \tr_{B}\rho_{ABCD} &= \rho_{ACD}= \Lambda_{AD}\Lambda_{CD} \text{ with }[\Lambda_{AD},\Lambda_{CD}]=0,\\
        \tr_{C}\rho_{ABCD} &= \rho_{ABD}= \Lambda_{AD}\Lambda_{BD} \text{ with }[\Lambda_{AD},\Lambda_{BD}]=0.
    \end{align*}
    In addition, there are compatibility conditions on the marginals which lead to 
    \[          \tr_B\rho_{ABD}=\tr_C\rho_{ACD},~ \tr_A\rho_{ABD}=\tr_C\rho_{BCD},~ \tr_A\rho_{ACD}=\tr_B\rho_{BCD}.
    \] 
    It is computationally very challenging to obtain defining equations of $\Q_G$ as Algorithm~\ref{alg:QCMI} would involve eliminating $60$ variables in a polynomial ring in $196$ variables, which is infeasible with current computational resources.
\end{example}

\begin{question}
    Is the variety $\Q_{G}$ irreducible for any tree $G$?
\end{question}

It would be desirable to find a parametrisation of $\Q_G$. Quantum information theory provides a map that recovers a unique quantum state compatible with given two-body marginals on a tree, called the \emph{Petz recovery map} \cite{hayden2004structure, petz1986sufficient}. However, our algebraic version of this map does not parametrise the QCMI variety $\Q_G$ due to the fact that we are working with complex symmetric matrices instead of Hermitian matrices. The Petz recovery map therefore gives rise to a different variety, which we introduce in the next subsection.

\subsection{Petz varieties} 
The Quantum Marginal Problem asks about how to reconstruct a quantum state of a composite system from the states of its subsystems. In the case of the 3-chain graph with ordered vertices $A,B$ and $C$, one can ask for a quantum state $\rho_{ABC}$ with given two-body marginals $\rho_{AB}$ and $\rho_{BC}$ and satisfying the quantum Markov condition $\qcmi{A}{C}{B}=0$. The answer to this particular problem is given by the Petz recovery map. This map is of algebraic nature and gives rise to an algebraic variety, the \emph{Petz variety}, which we study in this subsection. 

We start by introducing the Petz recovery map for the 3-chain graph with the associated Hilbert space $\H=\H_A\otimes\H_B\otimes\H_C$. Let $\C$ be the set of pairs of compatible invertible density operators on $\H_A\otimes\H_B$ and $\H_B\otimes\H_C$, respectively, i.e.\ an element in $\C$ is of the form $(\rho_{AB},\rho_{BC})$, where $\rho_{AB}$ and $\rho_{BC}$ are invertible density operators satisfying the compatibility condition $\tr_A\rho_{AB} = \tr_C\rho_{BC}$. The Petz recovery map $\R_G$ for the 3-chain graph $G$ is
\begin{equation*}
    \small \R_G\colon \C \rightarrow \D(\H_A\otimes\H_B\otimes\H_C)
\end{equation*}
\begin{equation}
    \label{equ:PetzMap}
    \small \R_G(\rho_{AB}, \rho_{BC}) = (\rho_{AB}^{1/2}\otimes \Id_C)(\Id_A\otimes\rho_{B}^{-1/2}\otimes\Id_C)(\Id_A\otimes\rho_{BC})(\Id_A\otimes\rho_{B}^{-1/2}\otimes\Id_C)(\rho_{AB}^{1/2}\otimes\Id_C),
\end{equation}
where $\Id_A,\Id_B$ and $\Id_C$ are the identity operators on $\H_A,\H_B$ and $\H_C$, respectively. The recovered state is compatible with the marginals and satisfies $I(A\colon C\mid B)=0$. Moreover, it is the unique maximiser of the von Neumann entropy among all states on $\H$ \cite[Thm.\ 1]{di2020recoverability}. This, in particular, shows that the map $\R'_G$ defined by
\begin{equation*}
    \small \R'_G(\rho_{AB}, \rho_{BC}) \!=\!\! (\Id_A\otimes\rho_{BC}^{1/2})(\Id_A\otimes\rho_{B}^{\minus 1/2}\otimes\Id_C)(\rho_{AB}\otimes\Id_C)(\Id_A\otimes\rho_{B}^{\minus1/2}\otimes\Id_C)(\Id_A\otimes\rho_{BC}^{1/2})
\end{equation*}
recovers the same state \cite[Rem.\ 2]{di2020recoverability}.\par 
The Petz recovery map \eqref{equ:PetzMap} gives rise to a rational map $R_G$. From this point forward we restrict to the qubit case $\H_A\cong\H_B\cong\H_C\cong \cc^2$; the general case is a straightforward generalisation. Let $\rho_{AB}^{1/2} = X = (x_{ij})$ be a $(4\times 4)$-symmetric matrix of variables $x_{11},x_{12},\dots,x_{44}$. In the same way, let $\rho_{BC}^{1/2} = Y = (y_{ij})$. Finally, let $\rho_B^{1/2} = Z = (z_{ij})$ be a $(2\times 2)$-symmetric matrix of variables. To reflect the required marginal compatibilities between $\rho_{AB}, \rho_{BC}$ and $\rho_{B}$, we impose the conditions $\tr_C(Y^2) = \tr_A(X^2) = Z^{2}$. These conditions cut out a variety $V$ in $\Sym^4_{\rr}\times \Sym^4_{\rr} \times \Sym^2_{\rr}$. In analogy to \eqref{equ:PetzMap}, the map $R_G: V \dashrightarrow\Sym^8_{\rr}$ sends a point $(x,y,z) \in V$ to 
\begin{equation}
    \label{equ:PetzRationalMap}
    (x\otimes \Id_C)(\Id_A\otimes z^{-1} \otimes \Id_C)(\Id_A\otimes y)(\Id_A\otimes z^{-1} \otimes \Id_C)(x\otimes \Id_C).
\end{equation}
We call the Zariski closure of $R_G(V)$ inside $\Sym^8$, the space of \emph{complex} symmetric $(8\times 8)$-matrices, the \emph{Petz variety} of $G$ and denote it~$\mathcal{P}_G$. 

\begin{remark}
    The expression \eqref{equ:PetzRationalMap} for $R_G$ gives a concrete polynomial parametrisation of $\P_G$.
    The polynomials appearing in $R_G$ have degree five and have a minimum of $20$ and maximum of $32$ terms. The number of parameter variables is 23 while the variety $V$ has dimension 17. Algorithm \ref{alg:Petz} provides a symbolic routine to compute the ideal of the Petz variety for arbitrary trees. When restricting to the case where $x,y$ and $z$ are diagonal, \eqref{equ:PetzRationalMap} gives yet another parametrisation of the classical graphical model of the 3-chain from the introduction.
\end{remark}

\begin{proposition}
    \label{prop:PetzIrred3Chain}
    Let $G$ be the 3-chain graph. The Petz variety $\P_G$ is irreducible.
\end{proposition}

\begin{proof}
    Consider the subset $S \subset \Sym^8_{\rr}\times\Sym^8_{\rr}$ consisting of pairs of invertible matrices whose partial traces agree. The condition that their partial traces agree defines a linear subspace of $\Sym^8_{\rr}\times\Sym^8_{\rr}$. Linear spaces are irreducible and taking out the locus of positive codimension where the matrices become singular preserves irreducibility. Therefore, $S$ is irreducible. Note that $\R_G$ can be considered as a map on $S$ and the Zariski closure of its image coincides with $\P_G$ since square roots and inverses of symmetric matrices are again symmetric. Therefore, since $\R_G$ is continuous, $\P_G$ is irreducible.
\end{proof}

The Petz map can be generalised to arbitrary trees by iteratively applying the procedure for 3-chains of a tree $G$ \cite{di2021complexity,di2020recoverability}. This is done by taking two leaves $v_1$ and $v_2$ of a tree $G$, and replacing $G\setminus\{v_1,v_2\}$ by a single vertex representing a joint state on this subgraph. The joint state on $G$ is then expressed in terms of states on $G\setminus\{v_1\}$ and $G\setminus\{v_2\}$ via \eqref{equ:PetzMap}; by applying this procedure iteratively to $G\setminus\{v_1\}$ and $G\setminus\{v_2\}$, we reduce to the level of two-body marginals. This process leads to a map as in \eqref{equ:PetzMap} involving only one- and two-body marginals; again, we denote the resulting map by $\R_G$. Note that the expression for $\R_G$ depends on the choice of $v_1$ and $v_2$ in each iteration.\par
We now generalise the construction of the Petz variety to arbitrary trees.

\begin{construction}[Petz variety]
    Let $G$ be a tree with $N$ vertices and let us fix an expression for $\R_G$ as obtained in the previous paragraph. Let $\varrho_1$ and $\varrho_2$ be the sets of one- and two-body marginals occurring in $\R_G$. Moreover, let $V$ be the variety inside $S=(\Sym^2_{\rr})^{\#\varrho_1} \times (\Sym^4_{\rr})^{\#\varrho_2}$ consisting of tuples of symmetric matrices satisfying compatibility constraints according to $G$. In analogy to \eqref{equ:PetzMap}, $\R_G$ gives rise to a rational map $R_G\colon V\dashrightarrow \Sym^{2^N}_{\rr}$. The \emph{Petz variety} $\P_G$ of $G$ is defined as $\overline{R_G(V)}\subseteq \Sym^{2^N}$. 
\end{construction}

Algorithm \ref{alg:Petz} makes this construction explicit and computes the ideal of $\P_G$.

\begin{algorithm}\label{alg:Petz}
    \caption{Computing the Petz variety $\P_G$}
    \hspace*{-18.5em} \textbf{Input:} A graph $G$ with $N$ vertices\\
    \hspace*{-10em} \textbf{Output:} An ideal defining the Petz variety $\P_G\subseteq \Sym^{2^N}$
    \begin{algorithmic}[1]
    \State $R_G\gets$ expression for $\R_G$ in terms of one- and two-body marginals
    \State $\varrho_1, \varrho_2 \gets$ sets of one- and two-body marginals, respectively, involved in $R_G$
    \For{all $\rho_{v}\in\varrho_1$}
        \State $Z_{v} \gets (z_{ij}^{v})$ symmetric $(2\times 2)$-matrix of variables
    \EndFor
    \For{all $\rho_{v_1v_2}\in\varrho_2$}
        \State $X_{v_1v_2} \gets (x_{ij}^{v_1v_2})$ symmetric $(4\times 4)$-matrix of variables
    \EndFor
    \State $S \gets (\Sym^2_{\rr})^{\#\varrho_1}_{\{Z_{v_i}\}}\times(\Sym^4_{\rr})^{\#\varrho_2}_{\{X_{v_iv_j}\}}$
    \State $\E\gets\emptyset$
    \For{all pairs $(\rho_{v_1v_2},\rho_{w_1w_2})\in\varrho_2^2$ such that $v_2=w_1$}
        \State $\E\gets \E\cup \{\text{entries of } \tr_{v_1}(X_{v_1v_2}^2) \minus \tr_{w_2}(X_{w_1w_2}^2)\} \cup \{\text{entries of } \tr_{v_1}(X_{v_1v_2}^2) \minus Z_{v_2}^2\}$
    \EndFor
    \State $V\gets$ variety defined by $\E$ inside $S$ \label{line:variety}
    \State $R_G\gets R_G$ with every $\rho_{v_1v_2}$ replaced by $X_{v_1v_2}$ and every $\rho_v$ replaced by $Z_v$
    \State \Return $\ker(R_G\colon \cc[\Sym^{2^N}]\rightarrow \cc[V])$
    \end{algorithmic}
\end{algorithm}

\begin{proposition}
    The Petz variety $\P_G$ does not depend on the choice of expression for the recovery map $\R_G$.
\end{proposition}

\begin{proof}
    By the same argument as for the 3-chain graph above, the map $\R_G$ does not depend on the chosen expression. Let $\C$ be the domain of $\R_G$; each element of $\C$ is a tuple consisting of $\#E$-many compatible invertible two-body marginals, where $E$ is the set of edges of $G$. Consider the set $(\Sym^4_{\rr})^{\#E}$ of $\#E$-tuples of real symmetric $(4\times 4)$-matrices; it is Zariski dense in the set of complex symmetric matrices. We have $\R_G(\C \cap (\Sym^4_{\rr})^{\#E}) = R_G(V \cap U)$ where $U\subseteq \cc^{2^N\times 2^N}$ is the Zariski dense open set of invertible matrices. The set 
    $R_G(V \cap U)$ is Zariski dense in the Petz variety $\P_G$. Let $\R'_G$ be another expression for the recovery map, then $\R_G(\C \cap (\Sym^4_{\rr})^{\#E})= R'_G(V\cap U)$ and denote by $\P'_G$ the variety defined by 
    $R'_{G}$. It follows that $R'_G(V\cap U)=R_G(V \cap U)$ so $\P_G$ and $\P'_G$ agree on a dense open set, thus $\P_G=\overline{R_G(V)} = \overline{R'_G(V)}=\P'_G$.
\end{proof}

\begin{proposition}
    For any tree $G$, the Petz variety $\P_G$ is irreducible.
\end{proposition}

\begin{proof}
    The proof is analogous to that of Proposition \ref{prop:PetzIrred3Chain}: $\P_G$ can be represented as the Zariski closure of the image of a linear space under a continuous map.
\end{proof}
\begin{remark}
    Computing the defining equations of the Petz variety is very challenging. Even in the case of the 3-chain graph, Algorithm \ref{alg:Petz} does not terminate as it involves symbolic computations in a polynomial ring with 59 variables. Applying numerical implicitisation techniques is also not straightforward for the same reason.
\end{remark}
\begin{remark}
    If we considered Hermitian matrices, the set of states recovered by the Petz map would coincide with the set of states satisfying SSA with equality \cite{hayden2004structure}. However, since the ambient space of our varieties is that of complex symmetric matrices, the QCMI variety~$\Q_G$ and the Petz variety $\P_G$ are not the same.
\end{remark}


\section{Quantum Graphical Models from Gibbs Manifolds}
\label{sec:QGMandGibbs}

 In this section, we consider a new class of quantum graphical models, which arise as Gibbs manifolds of families of Hamiltonians. These serve as examples of \emph{quantum exponential families} \cite{ZhouQEF}. Gibbs manifolds and Gibbs varieties have been introduced in \cite{pavlov2022gibbs} as tools to study convex optimisation and, in particular, quantum optimal transport from an algebro-geometric perspective; we recall the definition.

\begin{definition}
    Let $\L$ be a linear space of real symmetric matrices (LSSM) of size $n\times n$. The \emph{Gibbs manifold} $\gm(\L)$ of $\L$ is $\exp(\L)$ where $\exp$ denotes the matrix exponential. The \emph{Gibbs variety} $\gv(\L)$ of $\L$ is the Zariski closure of $\gm(\L)\subseteq \Sym^n$.
\end{definition}
In physics, the Gibbs manifold parametrises thermal states of a family of Hamiltonians. Those states are crucial to quantum many-body systems theory and computation \cite{alhambra2022}. The Gibbs manifold is not in general 
algebraic, while the Gibbs variety is. In special cases, their dimensions coincide. This happens, for instance, when $\mathcal{L}$ consists of diagonal matrices with rational entries. For such $\L$ the resulting Gibbs variety is toric \cite[Theorem 2.7]{pavlov2022gibbs} and recovers the classical notion of exponential families \cite{efron2022}. Moreover, the Gibbs manifold in this case is the intersection of the Gibbs variety with the cone of positive definite matrices. One advantage of considering the Gibbs variety instead of the Gibbs manifold
is that it becomes possible to treat related problems in quantum information by using symbolic and numerical methods from algebraic geometry. 

\subsection{Gibbs varieties of linear systems of Hamiltonians}
The quantum Ham-mersley--Clifford Theorem (Theorem \ref{thm:quHC}) suggests to consider exponentials of \emph{local} Hamiltonians, i.e.\ those that act non-trivially only on a small subsystem. However, we do not
consider the class of local \emph{commuting} Hamiltonians as they neither form an LSSM nor a unirational variety (see \S \ref{subsec:unirational}).\par 

To a simple, undirected graph $G=(V,E)$ we associate an LSSM of Hamiltonians as follows. For each clique $C$ in $G$, let 
$\L_{C} $ be the LSSM given by all Hamiltonians \emph{supported} on $C$, i.e.\ those that act nontrivially only on the tensor factors $\H_i$ such that $v_i\in C$ and act as identity on all other subsystems. More precisely, $\L_C=\otimes_i L_{C}^i$ where $L_{C}^i=\Sym^{d_i}$
for $v_i\in C$ and $L_{C}^i=\Id_{d_i}$ otherwise. The family of Hamiltonians associated to $G$ is then $\L_G=\sum_{C\in\C(G)}\L_C$ where the sum runs over all cliques of $G$ \cite[Equ.\ 17]{weis2023}. The quantum graphical model is $\gm(\L_G)$ intersected with the space of trace one matrices. The Gibbs variety $\gv(\L_G)$ gives an algebraic description of this model. 

\begin{example}
\label{eg:3chainLSSMGV}
    Consider the 3-chain graph $G$ and assume we are in the qubit case, i.e.\ $\H_A\cong\H_B\cong\H_C\cong\cc^2$. Then we have $\L_G = \Sym^2\otimes \Sym^2\otimes \Id_2 + \Id_2\otimes \Sym^2 \otimes \Sym^2$. Using numerical algebraic geometry techniques, we verify that no linear or quadratic equations vanish on $\gv(\L_G)$. The higher degree equations are not amenable to our computational techniques. 
    The dimension of $\gm(\L_G)$ is 15 while the dimension of $\gv(\L_G)$ is at most 22.
\end{example}

Gibbs varieties have a number of nice theoretical properties, e.g.\ under mild assumptions they are unirational
\cite[Thm.\ 3.6]{pavlov2022gibbs}. However, as seen in Example \ref{eg:3chainLSSMGV}, their defining ideals are
often difficult to compute. In view of this, one might hope to simplify the defining equations of the Gibbs variety by restricting the family of Hamiltonians to a subset inside $\L_G$. This approach is pursued in the next subsection.

\subsection{Gibbs varieties of unirational varieties}
\label{subsec:unirational}
A natural subset to consider inside $\L_G$ is $X_G \coloneqq \sum_{C\in\C(G)} X_C$, where $X_C$ is the set of decomposable tensors supported on $C$. Note that $X_G$ is not a linear space. However, it is still a unirational variety. This motivates the following extension of the notion of Gibbs varieties.

\begin{definition}
    Let $X$ be a unirational variety of symmetric matrices of size $n\times n$. The \emph{Gibbs variety} $\gv(X)$ of $X$ is the Zariski closure of $\exp(X)\subseteq \Sym^n$.
\end{definition}

A number of concepts related to Gibbs varieties of linear spaces generalise to the case of unirational varieties of symmetric matrices. If $X$ is unirational and has dimension $d$, then it can be parametrised by rational functions in $d$ variables $y_1,\ldots,y_d$. Therefore, one can think of $X$ as a single matrix with entries in $\mathbb{C}(y_1,\ldots,y_d)$. The eigenvalues of this matrix are elements of $\overline{\mathbb{C}(y_1,\ldots,y_d)}$ and will be referred to as the eigenvalues of $X$.
Let $A \in \mathbb{S}^n$, then the \emph{$X$-centraliser} of $A$ is the set $\mathcal{Z}_X(A) = \{B \in X \mid AB-BA=0\}$. We collect properties of Gibbs varieties of unirational varieties in the following statement.

\begin{proposition}

Let $X$ be a unirational variety of $(n\times n)$-symmetric matrices of dimension $d$. Let $m$ be the dimension of the $\mathbb{Q}$-linear space spanned by the eigenvalues of $X$ and $k$ be the dimension of the $X$-centraliser of a generic element of $X$. Then  $\dim(\mathrm{GV}(X)) =  m+d-k$. In particular, $\dim(\gv(X)) \leq n + d$. Moreover, if $X$ has distinct eigenvalues, then $\gv(X)$ is irreducible and unirational. 
\end{proposition}

\begin{proof}
    This proposition generalises \cite[Thm.\ 2.6]{pavlov2023logsparse} and \cite[Thm.\ 3.6]{pavlov2022gibbs}. Proofs of these statements carry over to the case of unirational varieties, since they only use the fact that an LSSM can be parametrised by rational functions in $y_1,\ldots, y_d$ and do not depend on these functions being linear.
\end{proof}

Note that symbolic \cite[Alg.\ 1]{pavlov2022gibbs} and numerical \cite[Alg.\ 1]{pavlov2023logsparse} implicitisation algorithms for Gibbs varieties generalise accordingly. 
\begin{example}
    Again, consider the 3-chain graph $G$ in the qubit case. The associated unirational variety is $X_G = \{K\otimes L\otimes \Id_2 + \Id_2\otimes M\otimes N \mid K,L,M,N\in\Sym^2\}\subseteq \Sym^8$. The dimension of $X$ is equal to $10$ inside the $36$-dimensional space of symmetric $(8\times 8)$-matrices. The Gibbs variety $\gv(X)$ is a $14$-dimensional irreducible variety cut out by nine linear forms and 66 quadratic equations in $\Sym^8$. These results were obtained by using techniques of numerical algebraic geometry. More precisely, we create a sample of points on $\gm(X)$ and then interpolate with polynomials of a fixed degree by setting up a Vandermonde matrix and computing its kernel via QR-decomposition to obtain a sparse representation, see \cite{breiding2018learning}; this procedure yields polynomials of degree one and two. As these equations cut out an irreducible variety of the correct dimension, we obtain a generating set of the prime ideal of $\gv(X)$.\par
    The equations we obtained exhibit a remarkably simple structure; e.g.\ all polynomials have coefficients $\pm 1$ and consist of at most eight terms. It would be very interesting to obtain a theoretical explanation of this phenomenon.
\end{example}

\section{Toric Varieties from Quantum Exponential Families Associated to Graphs}

In this section we explore a completely different family of Hamiltonians $\hh_G$ associated to a graph $G$. This results in quantum exponential families that have a richer structure than the ones considered in Section \ref{sec:QGMandGibbs}. However, it should be emphasised that, unlike the previous constructions, this is not a generalisation of classical graphical models. On the other hand, all the results in this section hold for general undirected graphs, not only trees.\par 

\subsection{Commuting Hamiltonians from graphs}
In quantum physics, to an undirected graph $G$ one associates an LSSM $\hh_G$ which gives rise to the definition of \emph{graph states}, used in the study of entanglement, e.g.\ \cite{graphStates}. More precisely, the graph state associated to $G$ is the \emph{stabiliser state} of $\hh_G$. Stabiliser states appear in the framework of the stabiliser formalism used in quantum error correction \cite[\S 10.5]{nielsen2002quantum}; in fact, all results in this section generalise to stabilisers. For more details, we provide an introduction to the stabiliser formalism in Appendix \ref{sec:StabForm}. However, instead of studying the graph states associated to $\hh_G$, here we focus on its Gibbs variety; this latter perspective gives yet another example of quantum exponential families.


Let $G=(V,E)$ be a graph with vertices $V=\{v_1,\dots,v_N\}$. To each vertex $v_i$, we associate a Hamiltonian $H_i = \bigotimes_{j=1}^N H_{i,j}$ with
\[
    H_{i,j} = \left\{\begin{array}{ll}
      \sigma_X   &  \text{ if } i=j,\\
      \sigma_Z   &  \text{ if } (i,j)\in E,\\
      \Id_2      &  \text{ else.}
    \end{array}\right.
\]
Here, $\sigma_X$ and $\sigma_Z$ are the \emph{Pauli-$X$ and Pauli-$Z$ matrices}
\[
    \sigma_X = \begin{pmatrix}
        0 & 1\\ 1 & 0\\
    \end{pmatrix},\quad
    \sigma_Z = \begin{pmatrix}
        1 & 0\\ 0 & -1\\
    \end{pmatrix}.
\]
Denote the linear span of this set of Hamiltonians by $\mathfrak{H}_G\coloneqq \langle H_i\mid i=1,\dots,N\rangle $. The Hamiltonians $H_i$ are elements of the \emph{Pauli group} $\P_N$ where
\[
    \P_1 \coloneqq \{\pm \Id_2, \pm i\Id_2, \pm \sigma_X, \pm i\sigma_X, \pm \sigma_Y, \pm i \sigma_Y, \pm \sigma_Z, \pm i\sigma_Z\}
\]
and $\P_N$ is the set of all $N$-fold tensor products of elements of $\P_1$ equipped with multiplication as the group operation. Here, $\sigma_Y$ denotes the Pauli-$Y$ matrix 
\[
    \sigma_Y = \begin{pmatrix}
        0 & -i\\
        i & 0\\
    \end{pmatrix}.
\]

\begin{example}   
    \label{eg:GraphStatesGraph}
    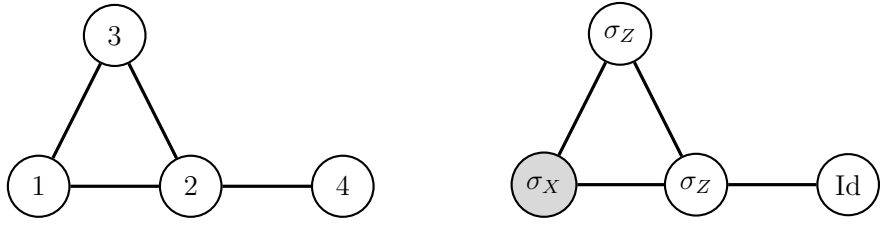
\begin{figure}
        \label{fig:graphState}
        \centering
        \begin{tikzpicture}
            \begin{scope}[every node/.style={circle,thick,draw,minimum size=0.9cm, scale=0.9}]
                \node (1) at (0,0) {1};
                \node (2) at (2,0) {2};
                \node (3) at (1,2) {3};
                \node (4) at (4,0) {4};
            \end{scope}
    
            \begin{scope}[
                  every edge/.style={draw=black,very thick}]
                \path [-] (1) edge (2);
                \path [-] (1) edge (3);
                \path [-] (3) edge (2);
                \path [-] (4) edge (2);
            \end{scope}
        \end{tikzpicture}\hspace{4em}
        ~~~~
        \begin{tikzpicture}
            \begin{scope}[every node/.style={circle,thick,draw,minimum size=0.9cm, scale=0.9}]
                \node[fill=gray!30!white] (1) at (0,0) {$\sigma_X$};
                \node (2) at (2,0) {$\sigma_Z$};
                \node (3) at (1,2) {$\sigma_Z$};
                \node (4) at (4,0) {$\Id$};
            \end{scope}
    
            \begin{scope}[
                  every edge/.style={draw=black,very thick}]
                \path [-] (1) edge (2);
                \path [-] (1) edge (3);
                \path [-] (3) edge (2);
                \path [-] (4) edge (2);
            \end{scope}
        \end{tikzpicture}

        \caption{Left: The graph $G$. Right: Illustration of the Hamiltonian $H_1$.}
    \end{figure}

    Consider the graph $G$ on four vertices depicted in Figure \ref{fig:graphState}. The Hamiltonian $H_1$ is given by
    \[
        H_1 = \sigma_X\otimes\sigma_Z\otimes\sigma_Z\otimes \Id_2.
    \]
    The linear space $\hh_G$ is spanned by the four Hamiltonians
    \[
       \sigma_X\otimes\sigma_Z\otimes\sigma_Z\otimes \Id_2, \sigma_Z\otimes\sigma_X\otimes\sigma_Z\otimes\sigma_Z, \sigma_Z\otimes\sigma_Z\otimes\sigma_X\otimes\Id_2, \Id_2\otimes\sigma_Z\otimes\Id_2\otimes\sigma_X.
    \]
\end{example}

In the following we consider the Gibbs variety of $\hh_G$. We start by showing that $\hh_G$ is a commuting family, implying that $\gv(\hh_G)$ is \emph{toric} after a linear change of coordinates and the Gibbs manifold $\gm(\hh_G)$ is semialgebraic \cite[Thm.\ 2.7]{pavlov2022gibbs}.

\begin{lemma}
    Any $H, H'\in\hh_G$ commute.
\end{lemma}

\begin{proof}
    W.l.o.g.\ assume $H$ and $H'$ are generators $H_m$ and $H_n$ of $\hh_G$. Note that the Pauli matrices satisfy the commutation relation
    \[
        [\sigma_j,\sigma_k] = 2i\epsilon_{jkl}\sigma_l,
    \]
    where $\epsilon_{jkl}$ is the Levi-Civita symbol and we denote $\sigma_1=\sigma_X,\sigma_2=\sigma_Y$ and $\sigma_3=\sigma_Z$. For $P\in\P_N$, let $\supp_X(P)\coloneqq \{j\in[N]\mid \sigma^{(j)}=\sigma_X\}$ and $\supp_Z(P)\coloneqq \{j\in[N]\mid \sigma^{(j)}=\sigma_Z\}$ denote the supports of $\sigma_X$ and $\sigma_Z$, respectively. Then two Pauli product matrices $P,Q\in\P_N$ containing only $\Id_2,\sigma_X$ or $\sigma_Z$ as tensor factors commute if and only if 
    \begin{equation}
        \label{equ:PauliComm}
         \#(\supp_X(P)\cap\supp_Z(Q)) + \#(\supp_Z(P)\cap\supp_X(Q)) \equiv 0  \mod 2.
    \end{equation}
    Let $\N(v)$ denote the set of neighbouring vertices of $v$ in $G$. Assume $v_n\in \N(v_m)$; then the left-hand side of (\ref{equ:PauliComm}) for $P=H_m$ and $Q=H_n$ becomes $\#\{m\} + \#\{n\}=2$. Finally, if $v_n\notin \N(v_m)\cup\{v_m\}$ the left-hand side of (\ref{equ:PauliComm}) is just zero.
\end{proof}

Let us recall from \cite{pavlov2022gibbs} how to obtain the toric variety and the coordinate change from $\gv(\hh_G)$. The symmetric matrices $H_1,\dots,H_N\in\hh_G$ are simultaneously diagonalisable,  i.e.\ there exist an orthogonal matrix $U$ and diagonal matrices $D_1,\dots,D_N$ such that $U^{-1}H_iU=D_i$ for $i=1,\dots,N$. The exponential of an element in $\hh_G$ can then be expressed as 
\begin{align*}
 \exp(x_1H_1+\dots +x_NH_N) = U\exp(x_1D_1+\dots +x_ND_N)U^{-1}   
\end{align*}
and thus $\gv(\hh_G) = U\cdot \gv(\mathfrak{D})\cdot U^{-1}$ where $\mathfrak{D}=\langle D_1,\dots,D_N\rangle$. 
Let $D_i = \diag(d_i), i = 1,\dots,N$ for $d_i\in \rr^{2^N}$ and let $\D = \langle d_1,\dots,d_N\rangle\subseteq \rr^{2^N}$ be the $\rr$-vector space spanned by the diagonals. 
Consider the smallest vector subspace $\D_{\qq}\subseteq\rr^{2^N}$ containing $\D$ that is spanned by elements of $\qq^{2^N}$ and choose an integral basis $a_1,\dots,a_N\in\zz^{2^N}$ of $\D_{\qq}$.
If $A$ denotes the $(N\times 2^N)$-matrix with rows $a_1,\dots,a_N$ then $\gv(\mathfrak{D})$ is the toric variety $X_A$ associated to $A$.\par 
We will need the following standard fact from the theory of quantum stabiliser codes. A proof is provided in Appendix \ref{sec:StabForm}.

\begin{lemma}[{\cite[Prop.\ 10.5]{nielsen2002quantum}}]
    \label{lem:eigPaulis}
    Let $S=\langle P_1,\dots,P_{N-k}\rangle$ be a subgroup of $\P_N$ generated  by $N-k$ independent\footnote{i.e.\ $\forall i=1,\dots,N-k\colon \langle P_1,\dots,\widehat{P_i},\dots,P_{N-k}\rangle\neq S$ where the hat denotes that this generator is omitted.} and commuting Pauli product matrices such that $-\Id_{2^N}\notin S$. Then the vector space $V_S\coloneqq \{v\in \rr^{2^N}\mid P_i v=v~\forall i=1,\dots,N-k\}$ of simultaneous $(+1)$-eigenvectors has dimension $2^k$.
\end{lemma}

\begin{remark}
    \label{rem:eigenspaceLemma}
    Note that any Pauli product matrix $P\in\P_N$ has eigenvalues $\pm 1$, both with multiplicity $2^{N-1}$ each. Lemma \ref{lem:eigPaulis} can then be rephrased as follows: the $(\pm 1)$-eigenspaces of $P_i$ intersect the eigenspaces of all $P_1,\dots,P_{i-1}$ in half their dimension. This fact is essential to the next theorem establishing a strong connection between quantum exponential families and classical algebraic statistics.
\end{remark}

\begin{theorem}
    \label{thm:mainThmGraphStates}
    For any graph $G$ with $N$ vertices, $\gv(\hh_G)$ is an independence model on $N$ binary random variables after a linear change of coordinates.
\end{theorem}

\begin{proof}
    As shown above, $\gv(\hh_G) = U\cdot X_A\cdot U^{-1}$ where the rows of $A$ are the diagonal entries of $D_i = U^{-1}H_i U$ for $i=1,\dots,N$ and $X_A$ is the (affine) toric variety associated to $A$. By Remark \ref{rem:eigenspaceLemma}, we can assume $A$ to be of the form 
    \begin{equation}
        \label{equ:matrixIndMod}
        A = \scalebox{0.8}{$
        \begin{pmatrix*}[r]
            -1 & -1 & -1 & -1 & \dots & -1 & -1 & -1 & -1 & 1 & 1 & 1 & 1 & \dots & 1 & 1 & 1 & 1 \\
            -1 & -1 & -1 & -1 & \dots & 1 & 1 & 1 & 1 & -1 & -1 & -1 & -1 & \dots & 1 & 1 & 1 & 1\\
            \vdots & \vdots & \vdots & \vdots & \ddots & \vdots& \vdots& \vdots& \vdots& \vdots& \vdots& \vdots& \vdots & \ddots & \vdots& \vdots& \vdots& \vdots\\
            -1 & -1 & 1 & 1 & \dots & -1 & -1 & 1 & 1 & -1 & -1 & 1 & 1 & \dots & -1 & -1 & 1 & 1\\
            -1 & 1 & -1 & 1 & \dots & -1 & 1 & -1 & 1 & -1 & 1 & -1 & 1 & \dots & -1 & 1 & -1 & 1\\
        \end{pmatrix*}
        $}
    \end{equation}
    i.e.\ the columns of $A$ are the vertices of the $N$-dimensional hypercube $[-1,1]^N$. Thus, $X_A$ is an independence model on $N$ binary random variables.
\end{proof}

\begin{remark}
    The variety $X_A$ above is not the independence model in its standard description. 
    For example, for $N=3$ we have
    \[
    A = \begin{pmatrix*}[r]
        -1 & -1 & -1 & -1 & 1 & 1 & 1 & 1\\
        -1 & -1 & 1 & 1 & -1 & -1 & 1 & 1\\
        -1 & 1 & -1 & 1 & -1 & 1 & -1 & 1\\
    \end{pmatrix*};
    \]
    the prime ideal of $X_A$ is 
    \begin{align*}
        I_A = & \langle x_1x_8 - 1, x_6x_7-x_5x_8, x_4x_7-x_3x_8, x_2x_7-1, x_4x_6-x_2x_8, \\
        & x_3x_6-1, x_4x_5-1, x_3x_5-x_1x_7, x_2x_5-x_1x_6, x_2x_3-x_1x_4 \rangle \subseteq \cc[x_1,\dots,x_8].
    \end{align*}
    The matrix $A'$
    of the independence model has as columns the vertices of  the hypercube $[0,1]^{N}$. Adding a row
    of ones to $A$ and to $A'$ yields the same variety. Thus, $X_A$ is an affine patch of the independence model, i.e.\ of the Segre variety $\sigma(\pp^1\times\pp^1\times\pp^1)\subseteq \pp^7$.
\end{remark}

A priori, it is not obvious how to obtain defining equations for $\gv(\hh_G)$ computationally in an efficient manner. However, Theorem \ref{thm:mainThmGraphStates} gives rise to Algorithm \ref{alg:graphStates} making computations of defining ideals for graphs with four or more vertices feasible.

\begin{algorithm}[h!]
    \caption{Computing defining equations of $\gv(\hh_G)$}
    \label{alg:graphStates}
    \hspace*{-25.5em} \textbf{Input:} A graph $G$\\
    \hspace*{-16.2em} \textbf{Output:} Polynomials defining $\gv(\hh_G)$
    \begin{algorithmic}[1]
    \State Compute $\hh_G=\langle H_1,\dots,H_N\rangle$
    \State $U \gets$ matrix simultaneously diagonalising $H_1,\dots,H_N$
    \State $I \gets$ ideal of the independence model defined by $A$ as in (\ref{equ:matrixIndMod}) in variables $p_{i,i}$ for $i=1,\dots,2^N$
    \State $Y = (y_{ij}) \gets$ linear coordinate change according to $UPU^{-1}$ where $P=(p_{ij})$
    \For{all generators $g_k$ of $I$}
        \State $h_k \gets g_k$ changed to $Y$-coordinates
    \EndFor
    \State $J \gets$ ideal generated by all $h_k$ and $y_{ij}=0$ for all $i\neq j\in\{1,\dots,2^N\}$
    \State \Return a set of generators of $J$
    \end{algorithmic}
\end{algorithm}

Note that Step 3 in Algorithm \ref{alg:graphStates} can be pre-computed.
This allows to reduce finding the equations of $\gv(\hh_G)$ to a linear 
algebra problem, therefore reducing the computational complexity. An implementation of this algorithm is available at \cite{mathrepo}.

\begin{example}
    Let $G$ be the graph from Example \ref{eg:GraphStatesGraph}. Using Algorithm \ref{alg:graphStates}, we compute that $\gv(\hh_G)\subseteq \Sym^{16}$ is defined by 296 quadratic equations in 136 variables. The average number of terms of each generator is about 1982. This highlights the fact that the equations defining the Gibbs variety can be quite involved. It would be impossible to compute these equations without using the additional structure
    of $\hh_G$: both \cite[Alg.\ 1]{pavlov2022gibbs} and \cite[Alg.\ 1]{pavlov2023logsparse} failed to compute this example.
\end{example}

\subsection{Quantum information projections}
Given an arbitrary quantum state $\rho$ we can ask for the member $\tilde{\rho}\in\Q$ of some quantum exponential family $\Q = \gm(\L)$ ``closest'' to $\rho$. Here, ``closest'' means minimising the \emph{quantum relative entropy}.

\begin{definition}
    The \emph{quantum relative entropy} $D(\rho || \sigma)$ between a state $\rho$ and a positive semidefinite operator $\sigma$ is 
    \[
        D(\rho || \sigma) \coloneqq \left\{\begin{array}{ll}
          \tr(\rho(\log(\rho) - \log(\sigma)) )  &  \text{ if } \supp(\rho)\subseteq \supp(\sigma)\\
          +\infty   &  \text{ otherwise.}
        \end{array}\right.
    \]
    Here, the support of a linear operator is the subspace orthogonal to the kernel with respect to the standard inner product on $\cc^n$, and all logarithms are taken to have base two.
\end{definition}
This is a quantum generalisation of the Kullback--Leibler divergence in classical information theory. Note that, similarly to the Kullback--Leibler divergence, the quantum relative entropy is \emph{not} an actual metric as it is not symmetric and does not satisfy the triangle inequality. However, it does satisfy non-negativity (quantum Gibbs' inequality). More precisely, if $\tr(\sigma)\leq 1$ we have $D(\rho || \sigma)\geq 0$ with equality if and only if $\rho=\sigma$. See \cite[\S 11.8]{wilde2013quantum} for an extensive reference.

\begin{definition} 
    \label{def:infprojection}
    The \emph{quantum information projection} $\tilde{\rho}$ of a quantum state $\rho$ to a quantum exponential family $\Q$ is the element of $\Q$ which is the closest to $\rho$ with respect to the quantum relative entropy
    \[
    \tilde{\rho}=\underset{\rho'\in \mathcal{Q}}{\argmin}~ D(\rho||\rho').
    \]
\end{definition}

The quantum information projection is unique and has been characterised in the case where $\Q$ consists of exponentials of local Hamiltonians \cite[Lem.\ 2]{niekamp2013computing}. Since the Gibbs manifold considered in the previous subsection is semialgebraic, we can use algebraic techniques to find the quantum information projection in this case. The following theorem gives an algebraic characterisation of the quantum information projection for a quantum exponential family $\Q$ of commuting Hamiltonians, in particular for $\Q = \gm(\hh_G)$. 

\begin{theorem}
    \label{thm:QIP}
    Let $\hh = \langle H_1,\dots, H_k\rangle$ be a linear span of commuting Hamiltonians in $\Sym^{d}_{\rr}$, fix a positive definite matrix $\rho\in \Sym^d_{\rr}$ and let $b_i \coloneqq \langle H_i,\rho\rangle = \tr(H_i\rho)$ for $i=1,\dots,k$. Let $M_{\rho}$ be the affine linear space defined by 
    \[
        M_{\rho} \coloneqq \{A\in \Sym^d_{\rr} \mid \langle H_i, A\rangle = b_i \text{ for } i=1,\dots,k\}.
    \]
    Then $M_{\rho}\cap \gm(\hh)$ consists of a unique point $\rho^*$. It is the maximiser of the von Neumann entropy inside $M_{\rho}$ and the quantum information projection of $\rho$ to $\gm(\hh)$.
\end{theorem}

\begin{remark}
    This result generalises Birch's Theorem \cite[Prop.\ 2.1.5]{berndLectures} to quantum exponential families that become toric after a linear change of coordinates.
\end{remark}

\begin{proof}
    The fact that $M_{\rho}\cap \gm(\hh) = \{\rho^*\}$ and $\rho^*$ is the unique point maximising the von Neumann entropy is a direct consequence of \cite[Thm.\ 5.1]{pavlov2022gibbs}. It remains to show that this point is the quantum information projection of $\rho$ to $\gm(\hh)$.\par 
    Let $\tilde{\rho}$ be the quantum information projection of $\rho$ to $\gm(\hh)$. As in the discussion preceding Lemma \ref{lem:eigPaulis}, let $U$ be the matrix diagonalising $\hh$ into $\dd = \langle D_1,\dots,D_k\rangle$, i.e.\ $H_i = UD_iU^{-1}$ for $i=1,\dots,k$, so $\tilde{\rho}\in U\gm(\dd)U^{-1}$. Minimising the quantum relative entropy between $\rho$ and $\gm(\hh)$ is then equivalent to minimising the quantum relative entropy between $U^{-1}\rho U$ and $\gm(\dd)$. Let $\sigma\coloneqq U^{-1}\rho U$; then we want to maximise $\tr(\sigma\log(\Delta))$ over diagonal matrices $\Delta\in\gm(\dd)$, i.e.\ find $\tilde{\rho}' = \diag(\hat{\delta})$ such that
    \[
        \tilde{\rho}' = \underset{\Delta\in\gm(\dd)}{\argmax}~ \sum_j \sigma_{jj}\log \Delta_{jj}.
    \]
    This is the same problem as finding the maximum likelihood estimator on the exponential family $\gm(\dd)$ given data $u = (\sigma_{11},\sigma_{22},\dots,\sigma_{dd})$. Note that every coordinate of $u$ is nonzero. By Birch's Theorem, $\hat{\delta}$ is the unique point on $\gm(\dd)$ satisfying $A\hat{\delta}=Au$ where $A$ is the matrix whose $i$th row is the diagonal of $D_i$ as in \eqref{equ:matrixIndMod}. Observe that 
    \[
        (Au)_i = \sum_j (D_i)_{jj} \sigma_{jj} = \tr(D_i \sigma) = \tr(D_i U^{-1}\rho U) = \tr(U D_i U^{-1}\rho) = \tr(H_i\rho) = b_i;
    \]
    analogously, $(A\hat{\delta})_i = \tr(H_i \tilde{\rho})$. This shows $\tilde{\rho}\in M_{\rho}\cap \gm(\hh)$ and thus $\tilde{\rho}=\rho^*$. 
\end{proof}

Theorem \ref{thm:QIP} provides a way to compute the quantum information projection to $\gm(\hh_G)$ algorithmically by using numerical algebraic geometry; concretely, one can first compute $M_{\rho}\cap \gv(\hh_G)$ and then choose the unique point lying in the PSD cone.

\begin{example}
    Consider the positive definite matrix
    \[
        \rho = \begin{pmatrix*}[r]
            84 & -22 & 11 & -51 & -15 & -8 & -26 & 4\\
            -22 & 51 & -5 & -7 & 23 & -13 & 17 & 40\\
            11 & -5 & 51 & 25 & -16 & -3 & 9 & 28\\
            -51 & -7 & 25 & 70 & -19 & 17 & 18 & -26\\
            -15 & 23 & -16 & -19 & 92 & 32 & 23 & 24\\
            -8 & -13 & -3 & 17 & 32 & 62 & 2 & -36\\
            -26 & 17 & 9 & 18 & 23 & 2 & 94 & 10\\
            4 & 40 & 28 & -26 & 24 & -36 & 10 & 109\\  
        \end{pmatrix*}
    \]
    and the 3-chain graph $G$. The intersection $M_{\rho}\cap \gv(\hh_G)$ consists of six real matrices. Only one of them is positive semidefinite, namely the matrix 
    \[
        \tilde{\rho} = \scalebox{0.9}{$\begin{pmatrix*}[r]
            20.5417  & \minus 12.5   &   \minus 20.5   &   \minus 12.4746  &  \minus 5.5     &   3.34685 &  \minus 5.48884 &  \minus 3.34006\\
            \minus 12.5   &    20.5417  &  12.4746  &  20.5    &    3.34685  &  \minus 5.5   &     3.34006  &  5.48884\\
            \minus 20.5   &    12.4746  &  20.5417  &  12.5   &     5.48884 &  \minus 3.34006 &   5.5    &    3.34685\\
            \minus 12.4746 &   20.5   &    12.5   &    20.5417  &   3.34006 &  \minus 5.48884 &   3.34685 &   5.5\\
            \minus 5.5     &   3.34685  &  5.48884  &  3.34006 &  20.5417  & \minus 12.5  &     20.5    &   12.4746\\
            3.34685 &  \minus 5.5    &   \minus 3.34006 &  \minus 5.48884 & \minus 12.5   &    20.5417   &\minus12.4746 &  \minus20.5\\
            \minus 5.48884  &  3.34006  &  5.5   &     3.34685 &  20.5   &   \minus12.4746  &  20.5417  &  12.5\\
            \minus 3.34006  &  5.48884  &  3.34685  &  5.5    &   12.4746 &  \minus20.5   &    12.5    &   20.5417\\
        \end{pmatrix*}$}.
    \]
    This matrix is the quantum information projection of $\rho$ to $\gm(\hh_G)$.
\end{example}

{\bf Acknowledgements}. 
E.D.\ was supported by the Funda\c{c}\~ao para a Ci\^encia e a Tecnologia (FCT) grant 2020.01933-CEECIND, and partially supported by CMUP under the FCT grant UIDB-00144-2020. The authors would like to thank Bernd Sturmfels for bringing the topic to their attention and Michael Wolf for useful discussions on graph states.

\bibliographystyle{plain}
\bibliography{bibliography}

\appendix 
\section{Stabiliser Formalism}
\label{sec:StabForm}

The purpose of this appendix is to provide a proof of Lemma \ref{lem:eigPaulis} to make the paper self-contained, and to generalise Theorem \ref{thm:mainThmGraphStates} by introducing the stabiliser formalism. This framework is commonly used in quantum error correction for a very convenient description of quantum code spaces. In our exposition we follow \cite[\S 10.5.1]{nielsen2002quantum}.\par 

Any subgroup $S\leq \P_{N}$ of the Pauli group acts on the vector space of $N$ qubit states by multiplication. The \emph{vector space stabilised by $S$} is denoted $V_S$ and we call $S$ the \emph{stabiliser} of $V_S$. In quantum error correction, $V_S$ is the code space. \par 

Let $S$ be generated by $S=\langle p_1,\dots,p_l\rangle$; the generators $p_1,\dots,p_l$ are called \emph{independent} if for all $i=1,\dots,l \colon \langle p_1,\dots,\hat{p_i},\dots,p_l \rangle \lneq S$, where the hat means that the element is omitted.

\begin{lemma}[{\cite[Prop.\ 10.5]{nielsen2002quantum}}]
    \label{lem:appStab}
    Let $S=\langle p_1,\dots,p_{N-k}\rangle \leq \P_N$ be generated by $N-k$ independent and commuting Pauli product matrices such that $-\Id_{2^N}\notin S$. Then $V_S$ has dimension $2^k$.
\end{lemma}

\begin{proof}
    First note that any Pauli matrix $\sigma\in\{\sigma_X,\sigma_Y,\sigma_Z\}$ has eigenvalues $\pm 1$, and the projector on the $\pm 1$-eigenspace of $\sigma$ is $\frac{\Id_2 \pm \sigma}{2}$. For any $\xx = (x_1,\dots,x_{N-k})\in (\zz/2\zz)^{N-k}$, define
    \[
        P^{\xx}_S \coloneqq \frac{1}{2^{N-k}}\prod_{j=1}^{N-k}(\Id_{2^N} + (-1)^{x_j}p_j);
    \]
    clearly, $P^{\mathbf{0}}_S$ is the projector onto $V_S$.
    \begin{claim}
    For any $\xx\in (\zz/2\zz)^{N-k},~ \dim(\im(P^{\xx}_S)) = \dim(\im(P^{\mathbf{0}}_S))$.
    \end{claim}
    We represent a Pauli product matrix $p\in \P_N$ as a (row) vector $v_p\in (\zz/2\zz)^{2N}$ as follows:
    \[
        (v_p)_i = \left\{\begin{array}{ll}
            1 & \text{ if $i\leq N$ and the $i$th tensor factor of }v_p \text{ is either } \sigma_X \text{ or } \sigma_Y, \\
            1 & \text{ if $i > N$ and the $(i-N)$th tensor factor of }v_p \text{ is either } \sigma_Z \text{ or } \sigma_Y, \\
            0 & \text{ else.}
        \end{array}\right.
    \]
    Then two Pauli product matrices $p,p'\in \P_N$ commute if and only if $v_p \Lambda v_{p'}^T=0$, where $\Lambda$ is the ($2N\times 2N$)-matrix 
    \[
        \Lambda = \begin{pmatrix}
            0 & \Id_N\\
            \Id_N & 0\\
        \end{pmatrix}.
    \]
    Let $p_1,\dots,p_{N-k}$ be the independent generators of $S$. For any $i=1,\dots,N-k$ there exists a $p\in \P_N$ such that $pp_ip^{\dagger} = -p_i$ and $pp_jp^{\dagger} = p_j$ for all $j\neq i$. Indeed, consider the ($(N-k)\times 2N$)-matrix $P$ with rows $v_{p_1},\dots,v_{p_{N-k}}$; as the generators are independent, one can check that the rows of $P$ are linearly independent. Therefore, the linear system $P\Lambda x = e_i$, where $e_i$ is the $i$th standard basis vector, has a solution, say $s\in (\zz/2\zz)^{2N}$. Then we define $p\in P_N$ by $v_p = s^T$. Thus, for any $j\neq i$ we have $v_{p_j}\Lambda v_p = 0$, so $p$ and $p_j$ commute and $pp_jp^{\dagger}=p_j$. Moreover, $v_{p_i}\Lambda v_p = 1$, hence $pp_ip^{\dagger}=-p_i$.\par 
    The argument above shows that for any $\xx\in (\zz/2\zz)^{N-k}$, there exists $p_\xx\in \P_N$ such that $P^{\xx}_S = p_\xx P^{\mathbf{0}}_S p_\xx^{\dagger}$, proving the claim.\par \bigskip
    Let $\xx,\xx'\in (\zz/2\zz)^{N-k}$ be two distinct vectors, i.e.\ there exists an $i\in \{1,\dots,N-k\}$ such that $\xx_i\neq \xx'_i$. Then $\im(P^{\xx}_S)$ and $\im(P^{\xx'}_S)$ are orthogonal. Indeed, the Hilbert--Schmidt inner product between $P^{\xx}_S$ and $P^{\xx'}_S$ evaluates to
    \[
        \langle P^{\xx}_S, P^{\xx'}_S \rangle = \frac{1}{2^{2(N-k)}} \tr \left((\Id+p_i)(\Id-p_i)\prod_{j\neq i}(\Id+(-1)^{x_j}p_j)(\Id+(-1)^{x'_j}p_j)\right) = 0
    \]
    as $(\Id+p_i)/2$ and $(\Id-p_i)/2$ are projectors on complementary eigenspaces of $p_i$.\par
    Finally, observe that
    \[
        \sum_{\xx\in (\zz/2\zz)^{N-k}} P^{\xx}_S =  \Id_{2^N},  
    \]
    so the $2^{N-k}$ many vector spaces $\im(P^{\xx}_S)$ form an equidimensional partition of $\cc^{2^N}$, hence $\dim(V_S) = \dim(\im(P^{\mathbf{0}}_S)) = 2^k$.
\end{proof}

It is now apparent that Theorem \ref{thm:mainThmGraphStates} does not rely on the structure of the graph but generalises to stabilisers.

\begin{theorem}
    Let $S=\langle p_1,\dots,p_N\rangle \leq \P_N$ be generated by $N$ independent and commuting Pauli product matrices such that $-\Id_{2^N}\notin S$. Then $\gv(S)$ is an independence model on $N$ binary random variables after a linear change of coordinates.
\end{theorem}

\begin{proof}
    The proof of Theorem \ref{thm:mainThmGraphStates} immediately generalises to this situation by using Lemma \ref{lem:appStab}.
\end{proof}

\end{document}